\documentclass[11pt,a4paper,leqno]{article}

\usepackage[utf8]{inputenc}
\usepackage[T1]{fontenc}       %
\usepackage[english]{babel}     %

\usepackage[pdftex]{graphicx}          

\usepackage{amsmath}
\usepackage{amssymb}
\usepackage{amsthm}
\usepackage{anysize}

\date{Wroc{\l}aw, \today}

\theoremstyle{plain}
\newtheorem{thm}{Theorem}[section]
\newtheorem*{thm*}{Theorem}
\newtheorem{lem}[thm]{Lemma}
\newtheorem{prop}[thm]{Proposition}
\newtheorem{clm}[thm]{Claim}
\newtheorem{cor}[thm]{Corollary}
\newtheorem{conj}[thm]{Conjecture}

\theoremstyle{definition}
\newtheorem{df}[thm]{Definition}
\newtheorem{denot1}[thm]{Notation}
\newtheorem{denot}[thm]{Notations}
\newtheorem{hypot}[thm]{Hypothesis}
\newtheorem{case}{Case}

\theoremstyle{remark}
\newtheorem{rem}[thm]{Remark}
\newtheorem{exmp}[thm]{Example}

\numberwithin{equation}{section}

\renewcommand{\Pr}{\mathrm {P}}
\newcommand{\ind}{{\mathbf {1}}}
\newcommand{\im}{{\mathrm{im}}}
\renewcommand{\SS}{{\mathbb{S}^2}}
\newcommand{\HH}{{\mathbb{H}^2}}
\newcommand{\HHH}{{\mathbb{H}^3}}
\newcommand{\Hn}{{\mathbb{H}^n}}
\newcommand{\op}{\omega^{(p)}}

\newcommand{\g}{\gamma}
\newcommand{\s}{\sigma}
\newcommand{\B}{\mathfrak{B}}

\newcommand{\ot}{{\tilde{o}}}
\newcommand{\rt}{{\tilde\varrho}}
\newcommand{\gs}{\gamma^*}
\newcommand{\gr}{\mathrm{gr}}
\newcommand{\F}{\mathrm{F}}

\newcommand{\oE}{{\vec{E}}}
\newcommand{\CG}{\Gamma}
\renewcommand{\d}{{^\dag}}
\newcommand{\rb}{{^\mathrm{rb}}}
\renewcommand{\r}{{\mathrm{r}}}
\newcommand{\R}{{\mathrm{R}}}
\newcommand{\fL}[1]{{f_{#1}^L}}
\newcommand{\ft}[1]{{f_{#1}^\triangle}}
\newcommand{\ud}{\mathrm{d}}
\newcommand{\dk}{\frac{\ud}{\ud k}}

\newcommand{\bd}{\partial\,}
\newcommand{\bdi}{\partial_\infty\,}

\newcommand{\cH}{{\widehat{\mathbb{H}}^3}}

\newcommand{\clcH}[1]{{\overline{#1}^\cH}}
\newcommand{\sm}{\setminus}
\newcommand{\sint}{\mathrm{int}\,}
\newcommand{\Isom}{\mathrm {Isom}}
\newcommand{\pc}{{p_\mathrm c}}
\newcommand{\pu}{{p_\mathrm u}}

\newcommand{\emd}[1]{\emph{#1}}
\newcommand{\work}{paper}
\newcommand{\partw}{paper}

\newcommand{\comment}[1]{}

\renewcommand{\{}{\left\lbrace}
\renewcommand{\}}{\right\rbrace}

\newlength{\cupdotwidth}
\begin{document}

\title{Non-uniqueness phase of Bernoulli percolation on reflection groups for some polyhedra in $\HHH$}


\author{Jan Czajkowski\\
\small Mathematical Institute, University of Wroc{\l}aw\\
\small pl. Grunwaldzki 2/4, 50-384 Wroc{\l}aw, Poland\\
\small E-mail: {\tt czajkow@math.uni.wroc.pl}}

\maketitle
\begin{abstract}
In the present paper I consider Cayley graphs of reflection groups of finite-sided Coxeter polyhedra in $3$-dimensional hyperbolic space $\HHH$, with standard sets of generators. As the main result, I prove the existence of non-trivial non-uniqueness phase of bond and site Bernoulli percolation on such graphs, i.e.~that $\pc<\pu$, for two classes of such polyhedra:
\begin{itemize}
\item for any $k$-hedra as above with $k\ge 13$;
\item for any compact right-angled polyhedra as above.
\end{itemize}
I also establish a natural lower bound for the growth rate of such Cayley graphs (when the number of faces of the polyhedron is $\ge 6$; see thm.~\ref{gr>=}) and an upper bound for the growth rate of the sequence $(\#\{\textrm{simple cycles of length $n$ through $o$}\})_n$ for a regular graph of degree $\ge 2$ with a distinguished vertex $o$, depending on its spectral radius (see thm.~\ref{thmgs} and rem.~\ref{g<=gs<=rt}), both used to prove the main result.
\end{abstract}
\section{Introduction}
Consider a connected, locally finite\footnote{That means: with finite degree of each vertex.} graph $\CG$. Fix $p\in[0;1]$ and for each edge, delete it from the graph with probability $1-p$ and leave it with probability $p$; do it independently for all edges of $\CG$. So we obtain a random subset of the set of edges of $\CG$, which forms, together with all the vertices of $\CG$, a random subgraph $\op$ of $\CG$. It is a model of percolation, which we call \emd{bond Bernoulli percolation} on $\CG$ \emd{with parameter} $p$. Now, consider the number of infinite components (called \emd{clusters}) of $\op$.
\begin{df}
We define the \emd{critical probability} for the above percolation model on $\CG$ as
$$\pc(\CG):=\inf\{p\in[0;1]:\Pr(\op\textrm{ has some infinite cluster})>0\}.$$
\end{df}
It turn out that, due to Kolmogorov 0-1 law, always $\Pr(\op\textrm{ has some infinite cluster})\in\{0;1\}$, so it is $0$ for $p<\pc(\CG)$ and $1$ for $p>\pc(\CG)$ (because the event in question is an \emd{increasing} event---see e.g.~sections 1.4 and 2.1 of \cite{Grim}).
Next, we define another threshold:
\begin{df}
\emd{Unification probability} for the above percolation model is
$$\pu(\CG):=\inf\{p\in[0;1]:\textrm{a.s.~there is a unique infinite cluster in $\op$}\}.$$
\end{df}
Obviously, $\pc(\CG)\le\pu(\CG)$.
\begin{df}
I call a graph $\CG$ \emd{vertex-transitive}, or \emd{transitive} for short, if its automorphism group acts transitively on the set of vertices of $\CG$. If there are just finitely many orbits of vertices of $\CG$ under its automorphism group, then I call $\CG$ \emd{quasi-transitive}.
\end{df}
If $\CG$ is connected, locally finite and transitive, then it is known by \cite[thm.~1]{NewmSchul} (see also \cite[thm.~7.5]{LP}) that the number of infinite clusters in $\op$ is a.s.~constant and equal to $0$, $1$ or $\infty$. Hence, for $p\in(\pc(\CG),\pu(\CG))$, a.s.~$\op$ has infinitely many components. Mathematicians are interested when this interval of $p$ is non-degenerate, i.e.~when $\pc(\CG)<\pu(\CG)$. Amenability plays an important role in these investigations:
\begin{df}
We call a locally finite graph $\CG$ \emd{non-amenable}, if
$$\inf \frac{|\bd K|}{|K|}>0$$
where the infimum is taken over all non-empty finite sets of vertices of $\CG$ and where $\bd K$ is the set of edges of $\CG$ with exactly one vertex in $K$. We call $\CG$ \emd{amenable}, if the infimum equals $0$.
\end{df}
A sufficient condition for $\pc(\CG)<\pu(\CG)$ has been conjectured by Benjamini and Schramm:
\begin{conj}[\cite{BS96}]\label{conjBS96}
If $\CG$ is a non-amenable, quasi-transitive\footnote{This notion is referred to as ``almost transitive'' in \cite{BS96}.} graph, then
$$\pc(\CG)<\pu(\CG).$$
\end{conj}
On the other hand, in case of a connected transitive graph, non-amenability is a necessary condition for that inequality---see \cite[thm.~7.6]{LP} and also an original paper \cite{BK89}.
\par There are several classes of graphs, for which the above conjecture has been established. First such result is due to \cite{GrimNewm} for the product of $\mathbb{Z}^d$ and an infinite regular tree of sufficiently high degree; then, among others, there is paper \cite{Lal} establishing that for Cayley graphs of a wide class of Fuchsian groups and paper \cite{BS}---for transitive, non-amenable, planar graphs with one end\footnote{A transitive graph has one end, if after throwing out any finite set of vertices it has still exactly one infinite component.}. There is also quite general result in \cite{PSN}: any finitely-generated non-amenable group has a Cayley graph $\CG$ with $\pc(\CG)<\pu(\CG)$, and an analogous result for a continuous model of percolation in \cite{Tyk}.
\par Many of those results are obtained as well for site Bernoulli percolation, which is defined in analogous way to the bond one, but we remove vertices at random (independently, with some fixed probability $p$) from the graph $\CG$ instead of edges; then, $\op$ is defined as the random subgraph of $\CG$ induced by the remaining vertices.\footnote{Actually, no qualitative difference between bond and site Bernoulli percolation seems to be there in proofs of such theorems as above.}
\par In the context of those results above for graphs arising naturally from the hyperbolic plane $\HH$, there is interest in seeking for concrete examples of graphs of natural tilings of $n$-dimensional hyperbolic space $\Hn$ for $n\ge 3$, for which we have $\pc<\pu$ for bond or site Bernoulli percolation.\footnote{It is known that such examples exist, when, roughly speaking, the degree of the graph is sufficiently high.} In this \work{} I find such examples in dimension $3$, arising from reflection groups in $\HHH$. Namely, the two main theorems are:
\begin{thm}\label{3phgen}
If $k\ge13$, then for $\CG$ the Cayley graph of a reflection group of a Coxeter $k$-hedron in $\HHH$ with the standard generating set,\footnote{Cf.~section \ref{secprelim} and hypothesis \ref{PiG}.} we have
$$\pc(\CG)<\pu(\CG)$$
for bond and site Bernoulli percolation on $\CG$.
\end{thm}
\begin{thm}\label{3phRACpt}
If $\Pi$ is a compact right-angled polyhedron in $\HHH$, then for $\CG$ the Cayley graph of reflection group of $\Pi$ with the standard generating set,\footnote{Cf.~section \ref{secprelim} and hypothesis \ref{PiG}.} we have
$$\pc(\CG)<\pu(\CG)$$
for bond and site Bernoulli percolation on $\CG$.
\end{thm}
Proofs of these theorems are completed in sections \ref{secgen} and \ref{secRACpt}, respectively, although they use some facts established in section \ref{secbasic} and the appendices.
\begin{rem}
In the setting of either thm.~\ref{3phgen} or thm.~\ref{3phRACpt} we have $\pc(\CG)\ge\frac{1}{k-1}>0$, which is well-known (see e.g.~\cite[thm.~1.33]{Grim} or \cite[prop.~7.11]{LP}). Also, e.g.~if $\Pi$ in that setting is compact, then $\CG$ is a Cayley graph of a finitely presented group with one end, so from \cite[thm.~10]{BB} we have $\pu(\CG)<1$ for bond percolation. So in  such situation, theorems \ref{3phgen} and \ref{3phRACpt} give $3$ phases of Bernoulli bond percolation on $\CG$, which provides a picture analogous to that from \cite{BS} for $2$-dimensional case (where also $3$ such non-trivial phases are established).
\end{rem}
In section \ref{secbasic} an introductory version of theorem \ref{3phgen} as well as \ref{3phRACpt} is proved:
\begin{thm}\label{3phbasic}
If $\Pi$ is a compact right-angled polyhedron in $\HHH$ with $k\ge 18$ faces, then for $\CG$ the Cayley graph of reflection group of $\Pi$ with the standard generating set,\footnote{Cf.~section \ref{secprelim} and hypothesis \ref{PiG}.} we have
$$\pc(\CG)<\pu(\CG)$$
for bond and site Bernoulli percolation on $\CG$.
\end{thm}
It uses relatively narrow set of facts, but it is less general than theorems \ref{3phgen}, \ref{3phRACpt}. More precise remarks on how particular tools used in this \partw{} can improve theorem \ref{3phbasic}, are in section \ref{secrems}. It is worth noting that the methods used in this \partw{} do not give inequality $\pc(\CG)<\pu(\CG)$ for some regular tiling of $\HH$, e.g.~with right-angled regular pentagons. In the appendices, two of those tools (theorems \ref{thmgs}, \ref{gr>=}), which are a bit interesting theorems themslves, are obtained.
\par All three theorems \ref{3phgen}, \ref{3phRACpt}, \ref{3phbasic} use estimates for $\rt(\CG)$ (see notations~\ref{Cas}) inspired by \cite{Nag} and based on Gabber's lemma (see subsection \ref{subsecGab}). However, the basis for proving them are the following bounds for $\pc$ and $\pu$.
\begin{df}\label{dfgr}
For a transitive graph $\CG$, I define the (\emd{upper}) \emd{growth rate} of $\CG$
$$\gr(\CG)=\limsup_{n\to\infty}(\#B(n))^{1/n},$$
where $B(n)$ is the closed ball of radius $n$ around some fixed vertex of $\CG$.
\end{df}
\begin{thm}[\protect{\cite[remarks in section 3]{Lyo95}}]\label{bdpc}
For bond and site Bernoulli percolation on any Cayley graph $\CG$,
$$\pc(\CG)\le\frac{1}{\gr(\CG)}.$$
\end{thm}
\begin{rem}
This theorem is present in the idea of proof of theorem saying that for Cayley graph of a group with exponential growth, $p_c<1$---see remarks in \cite[section 3]{Lyo95}. Also, a proof of that, where the inequality $\pc(\CG)\le1/\gr(\CG)$ is evident (formally, for Bernoulli bond percolation), is the proof of \cite[theorem 7.19]{LP}. Essentially, the same proof of \cite[theorem 7.19]{LP} is valid for site Bernoulli percolation, because in that case theorem 5.15 from that book (originally theorem 6.2 from \cite{Lyo90}) is valid as well, as the distribution of the cluster containing the root, for both bond and site Bernoulli percolation on a rooted tree, is exactly the same, if we only declare the root to be open in the case of site percolation, by an easy exercise.
\end{rem}
\begin{rem}\label{pathsinnonsmpl}
For graphs which are not simple (i.e.~admit loops and multiple edges), I always consider paths not as sequences of vertices, but rather as sequences of edges. In particular, I distinguish between any two paths passing through the same vertices, in the same order, but through different edges.
\end{rem}
\begin{denot}
For any graph $\CG$, not necessarily simple, $o$---its vertex and $n\in\mathbb{N}$, I define $a_n(\CG,o)$ to be the number of cycles (considered as paths; see the above remark) of length $n$ in $\CG$ starting from $o$, which do not use any vertex more than once. Then I put
$$\g(\CG,o) = \limsup_{n\to\infty}\sqrt[n]{a_n(\CG,o)}$$
(I may drop ``$o$'', when $\CG$ is transitive, because then it does not play any role).
\end{denot}
\begin{thm}[O.~Schramm, \protect{\cite[thm.~3.9]{Lyo00}}]\label{bdpu}
For bond and site Bernoulli percolation on any transitive graph $\CG$, we have
$$\pu(\CG)\ge\frac{1}{\g(\CG)}.$$
\end{thm}
\par A quite wide introduction to percolation theory is in \cite{Grim} and \cite{LP}. For theory of Coxeter groups and reflections groups, I used \cite{Dav} and \cite[chap.~IV]{Mask}. An introduction to hyperbolic spaces $\Hn$ is present in \cite{BH}, while an introduction to simplicial complexes---in \cite[chap.~I.7]{BH} and \cite[sect.~A.2]{Dav}.
\paragraph{Acknowledgements}
I express my gratitude to my advisor, Jan Dymara, for many important and helpful remarks and words of advice and for leading me through thos work. I am also grateful to Itai Benjamini and Ruth Kellerhals for a helpful correspondence.
\section{Preliminaries}\label{secprelim}
\begin{denot1}
In this \work{} I assume that $0\in\mathbb{N}$ (i.e.~$0$ is a natural number).
\end{denot1}
\begin{denot}\label{Cas}
Let $\CG$ be any graph, not necessarily simple, and $o$---its vertex. For any $n\in\mathbb{N}$, I define (in the context of remark \ref{pathsinnonsmpl})
\begin{itemize}
\item $C_n(\CG,o)$ to be the number of oriented cycles\footnote{Paths and cycles in graph theory can use some vertices and edges more than once} of length $n$ in $\CG$ starting from $o$;
\item $a^*_n(\CG,o)$ to be the number of oriented cycles of length $n$ in $\CG$ starting from $o$ without backtracks, i.e.~without pairs of consecutive passes through the same edge of $\CG$, forth and back. (I do not regard backtracks cyclically, i.e.~passing first and last the same edge is not a backtrack.)
\end{itemize}
Basing on that, I define:
\begin{itemize}
\item $\rt(\CG,o) = \limsup_{n\to\infty}\sqrt[n]{C_n(\CG,o)}$,
\item $\gs(\CG,o) = \limsup_{n\to\infty}\sqrt[n]{a^*_n(\CG,o)}$.
\end{itemize}
Sometimes I may drop ``$o$'' or ``$(\CG,o)$'' in all the above notations, if the context is obvious. Also, I will drop ``$o$'', when $\CG$ is transitive, because then the above values do not depend on $o$.
\end{denot}
\begin{rem}\label{g<=gs<=rt}
Note that for any $\CG$ and its vertex $o$,
$$a_n(\CG,o)\le a^*_n(\CG,o)\le C_n(\CG,o),$$
for $n\in\mathbb{N}$,\footnote{It is true for natural $n\ge 3$, to be strict.} hence also
$$\g(\CG,o) \le \gs(\CG,o) \le \rt(\CG,o).$$
It follows that, thanks to theorems \ref{bdpc} and \ref{bdpu}, it is sufficient to prove one of the inequalities $\gs(\CG)<\gr(\CG)$, $\rt(\CG)<\gr(\CG)$ in order to show that $\pc(\CG)<\pu(\CG)$.
\end{rem}
For definition and description of models of $\HHH$, see e.g.~chapters I.2 and I.6 of \cite{BH}.
\begin{df}\mbox{}
\begin{itemize}
\item By a half-space I mean by default a closed half-space.
\item By a \emd{convex polyhedron} in $\HHH$ I mean an intersection of finitely many half-spaces in $\HHH$ with non-empty interior.
\item For a plane supporting a convex polyhedron, i.e.~intersecting it so that it is left on one side of the plane, I call that intersection a \emd{face}, an \emd{edge} or a \emd{vertex}, if it is $2$-, $1$- or $0$-dimensional, respectively.
\item Two distinct faces $F_1$, $F_2$ of polyhedron are \emd{neighbours} (or: do \emd{neighbour}), iff they have an edge in common; we then write $F_1\sim F_2$.
\item The \emd{degree} of a face $F$ of a polyhedron, denoted by $\deg(F)$, is the number of its neighbours.
\item A convex polyhedron is \emd{Coxeter} iff the dihedral angle at each edge of it is of the form $\pi/m$, where $m\ge 2$ is natural.
\end{itemize}
\end{df}
\begin{denot}
For a graph $\CG$, by $V(\CG)$, $E(\CG)$ I denote the sets of vertices and edges of $\CG$, respectively.
\end{denot}
\section{Description of the polyhedron and the graph}\label{secPiG}
The graphs that I am going to consider in this \work{} are Cayley graphs of reflection groups of Coxeter polyhedra in $\HHH$. Below I introduce these notions:
\begin{df}
The \emd{reflection group} of a convex polyhedron $\Pi$ is the subgroup of the whole group $\Isom(\HHH)$ of isometries of $\HHH$, generated by the hyperbolic reflections in planes of faces of $\Pi$.
\par A \emd{Coxeter matrix} over a finite set $S$ is a matrix $(m(s,t))_{s,t\in S}$ with elements in $\mathbb{N}\cup\{\infty\}$ such that
\begin{equation*}
m(s,t)=\{
\begin{array}{rl}
1,&\textrm{ if }s=t\\
\ge 2,&\textrm{ otherwise.}
\end{array}\right.
\end{equation*}
A \emd{Coxeter system} is a pair $(G,S)$, where $G$ is a group generated by $S$ with the presentation
\begin{equation*}
G=\langle S|s^2:s\in S; (st)^{m(s,t)}: s,t\in S, m(s,t)\neq\infty\rangle,
\end{equation*}
where $(m(s,t))_{s,t\in S}$ is a Coxeter matrix. We then call $G$ a \emd{Coxeter group}.
\end{df}
\begin{rem}
If $\Pi$ is a Coxeter polyhedron in $\HHH$, $S$ is the set of reflections in planes of its faces and $G=\langle S\rangle$ is its reflection group, then $(G,S)$ is a Coxeter system. The underlying Coxeter matrix $(m(s,t))_{s,t\in S}$ is connected with the dihedral angles of $\Pi$ as follows: for neighbouring faces $F_1,F_2$ of $\Pi$ and $s_1,s_2\in S$---the corresponding generators (respectively)---the dihedral angle between these faces equals $\pi/m(s_1,s_2)$. It is a special case of Poicar\'e's Polyhedron Theorem---theorem IV.H.11 in \cite{Mask} (one has to use also the proposition in IV.I.6 there).
Moreover, from that theorem, $\Pi$ (or $\sint\Pi$, depending on the convention) is a fundamental polyhedron for $G$ (see \cite[defnition IV.F.2]{Mask}; roughly speaking, that is why the orbit $G\cdot\Pi$ is called a \emd{tiling} of $\HHH$ and $v\Pi,v\in G$---its \emd{tiles}) and because of that, the only stabiliser in $G$ of any interior point of $\Pi$ is the identity. Hence, fixing $o\in\sint\Pi$, one can identify $G$ with the orbit of $o$ by the obvious bijection. (Cf.~also \cite[theorem 6.4.3]{Dav} for compact case.)
\end{rem}
\begin{hypot}\label{PiG}
For the rest of this \work{}, let $\Pi$ be a Coxeter polyhedron in $\HHH$, let $S$ be the set of reflections in planes of its faces and let $G=\langle S\rangle$ be its reflection group. Let $\CG$ be the Cayley graph of the Coxeter system $(G,S)$, i.e.~the undirected (simple) graph with $V(\CG)=G$ such that an edge joins two vertices $v,w$ in $\CG$ iff $v=ws$ for some $s\in S$. I am going to think of $G=V(\CG)$ as of the orbit of a fixed point $o\in\sint\Pi$, so that $V(\Pi)\subseteq\HHH$, and $o$ is the identity element of $G$. (These hypotheses are valid unless indicated otherwise.)
\end{hypot}
\begin{rem}
The graph $\CG$ is dual to the tiling $G\cdot\Pi$ in the sense that each $v\in V(\CG)=G$ naturally corresponds with the tile $v\Pi$ (hence one may denote $v\Pi$ by $v\d$) and an edge $\{v,w\}\in E(\CG)$ indicates that the tiles $v\Pi$ and $w\Pi$ have a face in common. Indeed, the tiles $v^{-1}v\Pi=\Pi$ and $v^{-1}w\Pi$ share a face, say the face of $\Pi$ corresponding to $s$ for some $s\in S$, iff $s=v^{-1}w$).
\end{rem}
\begin{denot1}\label{ed}
I use the above remark to introduce the following notion: for $e=\{v,w\}\in E(\CG)$ (or $e=(v,w)\in\oE(\CG)$---see notation \ref{oE}), let $e\d$ be the common face of $v\Pi$ and $w\Pi$.
Then, we say that $e$ traverses or passes through the face $e\d$.\footnote{Though, the geodesic segment between the endpoints of $e$ (in $\HHH$) does not need to literally cross the face $e\d$, so the introduced notion has a combinatorial nature.}
\end{denot1}
\section{A basic approach}\label{secbasic}
In this section I present a basic approach to proving the existence of non-trivial middle phase of Bernoulli bond and site percolation on some cocompact reflection groups of $\HHH$. The price for its relative easiness, in comparison to more developed versions I present later in this \partw{}, is that it works only for cocompact right-angled reflection groups with at least $18$ generators. Still, there is a quite elegant example of such group:
\begin{exmp}
Let $\Pi$ be a hyperbolic right-angled truncated icosahedron (i.e.~with the combinatorics of a classical football) in $\HHH$. Then its reflection group is a cocompact right-angled reflection group with $32$ generators.
\end{exmp}
\begin{thm*}[recalled theorem \ref{3phbasic}]
If $\Pi$ is a compact right-angled polyhedron in $\HHH$ with $k\ge 18$ faces, then for $\CG$ the Cayley graph of reflection group of $\Pi$ with the standard generating set, we have
$$\pc(\CG)<\pu(\CG)$$
for bond and site Bernoulli percolation on $\CG$.
\end{thm*}
Before I prove this theorem, I establish some general facts about reflection groups in $\HHH$, useful in the whole \work{}.
\subsection{Gabber's lemma}
\label{subsecGab}
I estimate $\rt$ of a reflection group of a hyperbolic polyhedron using Gabber's lemma.
\begin{df}\label{sprad}
For any (simple) graph $\CG$, by its \emd{spectral radius}, denoted by $\varrho(\CG)$, I mean the spectral radius of the simple random walk on $\CG$ starting from its fixed vertex $o$\footnote{The simple random walk on $\CG$ starting from $o$ is a Markov chain with state space $V(\CG)$, where at each step probability of transition from vertex $x$ to its neighbour is $1/\deg(x)$---see \protect{\cite[section I.1]{Woe}}}, i.e.
\begin{equation*}
\varrho(\CG) = \limsup_{n\to\infty}\sqrt[n]{p^{(n)}(o,o)},
\end{equation*}
where for natural $n$, $p^{(n)}(o,o)$ denotes the probability that the simple random walk on $\CG$ starting from $o$ is back at $o$ after $n$ steps. Note that spectral radius of $\CG$ does not depend on the choice of $o$ (as for $o,o'\in V(\CG)$, there is $C>0$ such that for $n\in\mathbb{N}$, $p^{(n+\mathrm{dist}(o,o'))}(o',o') \ge Cp^{(n)}(o,o)$ and \emph{vice-versa}, whence $p^{(n)}(o,o)$ and $p^{(n)}(o',o')$ has the same asymptotic behaviour).
\end{df}
\begin{rem}\label{remrrt}
It is easily seen that for any (simple) $k$-regular graph $\CG$ and for $n\in\mathbb{N}$, $p^{(n)}(o,o)=C_n(\CG,o)/k^n$, so $\varrho(\CG)=\rt(\CG)/k$.
\end{rem}
\begin{denot}\label{oE}
For an undirected (simple) graph $\CG$, let $\oE(\CG)$ denote the set of edges of $\CG$ given orientations. (It can be formalised by $\oE(\CG)=\{(v,w):\{v,w\}\in E(\CG)\}$.) For $e\in\oE(\CG)$, let $\bar{e}$ denote the edge inverse to $e$ and let $e_+$,$e_-$ denote the end (head) and the origin (tail) of $e$, respectively.
\end{denot}
\begin{lem}[Gabber, see \protect{\cite[prop.~1]{Bart}\footnote{They formulate this lemma for regular graphs, but their proof is valid in the below generality. Also, cf.~\cite[lemma of Gabber, p.~2]{Nag}.}} for the proof]
Let $\CG$ be an unoriented, infinite, locally finite graph and let a function $F:\oE(\CG)\to\mathbb{R}_+$ satisfy
$$F(e)=F(\bar{e})^{-1}$$
for each edge $e\in\oE(\CG)$. If there exists a constant $C_F>0$ such that for each vertex $v$ of $\CG$,
$$\frac{1}{\deg(v)}\sum_{e_+=v} F(e)\le C_F,$$
then
$$\varrho(\CG)\le C_F.$$
\end{lem}
In this \partw{} I am going to use the notion of $\rt$ rather than $\varrho$, as the former seems more convenient here. So I reformulate the above lemma, using remark \ref{remrrt}:
\begin{cor}\label{Gabcor}
Let $\CG$ and $F$ be as in the above lemma and assume that $\CG$ is regular. Then, if a constant $C_F>0$ satisfies
$$\sum_{e_+=v} F(e)\le C_F$$
for each vertex $v$ of $\CG$, then
$$\rt(\CG)\le C_F.$$
\end{cor}
\begin{df}\label{OG}\label{lgth}
Let $\CG$ be Cayley graph of any Coxeter system $(G,S)$ and $o$ be the identity element of $G$. I denote by $l$ the \emd{length function} on $(G,S)$, i.e.~for $v\in V(\CG)$, $l(v)$ is the graph-theoretic distance from $v$ to $o$ (or: the least length of a word over $S$ equal $v$ in $G$). Let $O(\CG)$ denote the \emd{standard orientation} of edges of $\CG$ arising from $l$:
\begin{equation}
O(\CG):=\{e\in\oE(\CG):l(e_+)>l(e_-)\}.
\end{equation}
\end{df}
\begin{rem}\label{l>l}
Note that in the setting of definition \ref{lgth}, we never have $l(e_+)=l(e_-)$ for $e\in\oE(\CG)$, because otherwise we would obtain a word over $S$ of odd length, trivial in $G$, hence a product of conjugates of the Coxeter relations, which are of even lengths, a contradiction.
\par So, for $e\in E(\CG)$, exactly one of the two oriented edges corresponding to $e$ is in $O(\CG)$.
\end{rem}
\begin{df}
From now on, whenever I mention or use the orientation of an unoriented edge of $\CG$, I mean, by default, the orientation defined by $O(\CG)$. Particularly, by an edge of $\CG$ \emd{passing to} (or \emd{from}) vertex $v\in V(\CG)$ I mean an edge $e$ from $E(\CG)$ with $e_+=v$ (respectively $e_-=v$), when oriented according to $O(\CG)$.
\end{df}
\begin{rem}
Assume hypothesis \ref{PiG} (particularly, on $\CG$). Whenever $e\in \oE(\CG)$, the plane $P$ of $e\d$ separates the endvertices of $e$, and separates one of them from $o$---let us call it $v$ and the other endvertex---$u$. Then, it appears that $(u,v)\in O(\CG)$. To see that,
take a path $p$ in $\CG$ from $o$ to $v$ of minimal length: $l(v)$. Here, consider $p$ as a sequence of vertices. Then, reflect in $P$
every fragment of $p$ lying outside $P$ (from the point of view of $o$), obtaining a path 
(which stays in one vertex for some steps, such steps contributing $0$ to the path length) from $o$ to $u$ with the length strictly less than $l(v)$. 
Hence a geodesic in $\CG$ from $o$ to $u$ has length $l(u)<l(v)$, so $(u,v)\in O(\CG)$.
\end{rem}
\begin{denot}
For $v\in V(\CG)$, in the setting of definition \ref{OG}, let $r(v)$ denote the number of edges of $\CG$ passing to $v$, i.e.
$$r(v)=\#\{e\in O(\CG):e_+=v\}.$$
Note that $r(v)>0$ for $v\neq o$ (and $r(o)=0$).
For natural $i$, let $q_i(v)$ be the number of edges passing from $v$ to vertices with $r(\cdot)=i$; formally
$$q_i(v)=\#\{e\in O(\CG):e_-=v,r(e_+)=i\}.$$
Note that always $q_0(v)=0$. I will prove in proposition \ref{r<=3} that in the setting of hypothesis \ref{PiG} we have $r(v)\le 3$ (whence $q_i(v)=0$ for $i>3$).
\end{denot}
\subsection{Geometry of $\Pi$ and $\CG$}\label{geomPiG}
The only assumption on $\Pi$ in this subsection is hypothesis \ref{PiG}. Here I am going to use some geometrical facts on the tiling of $\HHH$ for estimating $r(v)$ and for using corollary \ref{Gabcor} (see the previous subsection):
\begin{prop}\label{ePi}
For any two faces of $\Pi$, either they neighbour or they lie in disjoint planes.
\end{prop}
\begin{prop}\label{vPi}
For any three faces of $\Pi$ whose planes have non-empty intersection, some vertex of $\Pi$ belongs to all those faces.
\end{prop}
Those facts are easy consequences of the theorem from \cite{Andr}. I formulate it below in a version for $\HHH$ and a finite-sided convex polyhedron (just as in hypothesis \ref{PiG}) for simplicity. Before, I need some definitions:
\begin{df}\label{dffromAndr}
Consider the Klein unit ball model of $\HHH$ with its ideal boundary $\bdi\HHH$, which is the unit sphere in $\mathbb{R}^3$. Then for $A\subseteq\HHH$, we denote by $\clcH{A}$ its closure in $\cH:=\HHH\cup\bdi\HHH$. Then, a single point in $\bdi\HHH$ is claimed to have dimension $-1$ (as opposed to a point in $\HHH$) and the empty set---to have dimension $-\infty$.
\end{df}
\begin{thm}[a version of the theorem from \cite{Andr}]
Let $\Pi$ be a convex polyhedron in $\HHH$ (as in hypothesis \ref{PiG}) whose all dihedral angles are non-obtuse and let $(F_i)_{i\in I}$ be any family of its faces and for $i\in I$, $P_i$ be the plane of $F_i$. Then
$$\dim\bigcap_{i\in I}\clcH{F_i} = \dim\bigcap_{i\in I}\clcH{P_i}.$$
\end{thm}
\begin{rem}
In order to conclude propositions \ref{ePi} and \ref{vPi} from the above theorem, one has to observe that
$$\bigcap_{i\in I}F_i\neq\emptyset \Leftrightarrow \dim\bigcap_{i\in I}\clcH{F_i}\ge 0 \Leftrightarrow \dim\bigcap_{i\in I}\clcH{P_i}\ge 0 \Leftrightarrow \bigcap_{i\in I}P_i\neq\emptyset$$
using definition \ref{dffromAndr} and the convexity of $\bigcap_{i\in I}\clcH{F_i}$ and $\bigcap_{i\in I}\clcH{P_i}$ in the Klein model. That immediately gives proposition \ref{vPi}. To conclude proposition  \ref{ePi} as well, observe that if two faces of $\Pi$ intersect, then they have at least common vertex. Then, they must have a common edge, because otherwise the polyhedral angle of $\Pi$ at their common vertex would have more than three faces, which is impossible for a polyhedron with non-obtuse dihedral angles. So we obtain proposition \ref{vPi}.
\end{rem}
\begin{rem} 
By default, I consider the planes of faces of $\Pi$ oriented outside the polyhedron, i.e. the angle between them is $\pi$ minus the angle between their normal vectors, which I consider always pointing outside the polyhedron.
\end{rem}
\begin{prop}\label{v:r>=3}
For any $v\in V(\CG)$ and any three edges passing to $v$, they correspond to three faces of the tile $v\Pi$ which have a vertex in common.
\end{prop}
\begin{proof}
Assume that some three edges pass to a vertex $v$. Let $H_1,H_2,H_3$ be the (closed) half-spaces containing $v\cdot\Pi$ whose boundaries contain the faces $F_1,F_2,F_3$, respectively, corresponding to those edges. Because none of $H_1,H_2,H_3$ contains $o$ and none of them contains another, their boundaries must intersect pairwise. Further, if those all boundaries had no common point, the intersection $A=H_1\cap H_2\cap H_3$ would be a bi-infinite triangular prism and $H_1$ would contain $\HHH\sm (H_2\cup H_3)$ together with $o$, which contradicts the assumptions. So the planes of $F_1,F_2,F_3$ have a point in common---call it $p$ and by proposition \ref{vPi}, $F_1,F_2,F_3$ share $p$ as a common vertex.
\end{proof}
\begin{prop}\label{r<=3}
For any $v\in V(\CG)$, we have $r(v)\le 3$.
\end{prop}
\begin{proof}
Assume \emph{a contrario} that $r(v)\ge 4$. Then there are half-spaces $H_1,H_2,H_3,H_4$ corresponding to faces $F_1,F_2,F_3,F_4$ of $\Pi$ and separating $o$ from $\Pi$. Let $A=H_1\cap H_2\cap H_3$ (we know that it is a trihedral angle from the proof of proposition \ref{v:r>=3}).
\par $\bd H_4$ has to cross $\sint A$ in order to produce face $F_4$ and has to cross each of the edges of $A$ in exactly $1$ point other than $p$ (from proposition \ref{v:r>=3} and the fact that $H_4$ contains all the edges of $\Pi$ adjacent to $p$).
Note that $p\in\sint H_4$.
It is clear that the open half-line $po\sm\overrightarrow{po}$ of the line $po$ lies in $A$, so it crosses $\bd H_4$. Hence, $\overrightarrow{po}\subseteq H_4$ along with $o$, a contradiction.
\end{proof}
\begin{denot}
Now, I define the function $F$ for use of corollary \ref{Gabcor}: for $e\in O(\CG)$, let
$$F(e)=c_{r(e_+)},\quad F(\bar e)=\frac{1}{c_{r(e_+)}},$$
according to the condition from the corollary. Here $c_1,c_2,c_3>0$ are parameters (only three ones, as here always $0<r(e_+)\le 3$). I will write $(c_1,c_2,c_3)=\bar c$ for short. Let for $v\in V(\CG)$,
$$f_v(\bar c)=\sum_{e_+=v} F(e).$$
\end{denot}
In this setting
\begin{align}
f_v(\bar c)
&=\sum_{e\in O(\CG):e_+=v} c_{r(v)} + \sum_{e\in O(\CG):e_-=v} \frac{1}{c_{r(e_+)}}= \label{f(r,q)1}\\
&=r(v)c_{r(v)} + \sum_{i=1}^3 \frac{q_i(v)}{c_i},\label{f(r,q)2}
\end{align}
because $q_i(v)=0$ for $i>3$ (due to proposition \ref{r<=3}) and for $i=0$.
Recall that $k$ is the number of faces of $\Pi$.
\begin{lem}\label{rhobasic}
If $k\ge 6$, we have
$$\rt(\CG)\le 2\sqrt{3(k-3)}.$$
\end{lem}
\begin{rem}
The above bound has a better asymptotic behaviour that those in lemmas \ref{rhogen} and \ref{rhoRACpt}, but the latter ones give the inequality $\pc<\pu$ for a bit more values of $k$ (see section 
\ref{secrems}).
\end{rem}
\begin{proof}
I am going to choose values of $c_1,c_2,c_3$ giving a good upper bound for $\sup_{v\in V(\CG)}f_v(\bar c)$: let $c_1=c_2=c_3=\sqrt{\frac{k-3}{3}}$. Then from \eqref{f(r,q)1}
\begin{align*}
f_v(\bar c) &= r(v){\textstyle\sqrt{\frac{k-3}{3}}} + (k-r(v)){\textstyle\sqrt{\frac{3}{k-3}}}\le\\
&\le 3{\textstyle\sqrt{\frac{k-3}{3}}} + (k-3){\textstyle\sqrt{\frac{3}{k-3}}}=2\sqrt{3(k-3)},
\end{align*}
because here $\sqrt{\frac{k-3}{3}}\ge\sqrt{\frac{3}{k-3}}$ and $r(v)\le 3$.
\end{proof}
\begin{df}\label{nerve}
For a Coxeter system $(G,S)$, we call a subset $T\subseteq S$ \emd{spherical}, if the subgroup $\langle T\rangle$ is finite. The \emd{nerve} of $(G,S)$ is the abstract simplicial complex\footnote{For definitions of abstract simplicial complex and its geometric realisation, see \cite[section A.2]{Dav} or chapter I.7, especially the appendix of chapter I.7, of \cite{BH}.} whose simplices are all non-empty spherical subsets of $S$.\footnote{For a definition of nerve, see also \cite[section 7.1.]{Dav}.} 
\end{df}
\begin{clm}\label{no4inL}
In the setting of hypothesis \ref{PiG}, no subset of $S$ of cardinality $4$ is spherical (i.e.~the nerve of $(G,S)$ contains no $3$-simplex).
\end{clm}
\begin{proof}
First, note that any $3$ generators constituting a spherical subset of $S$, correspond to reflection planes with non-empty intersection (otherwise those planes would be the planes of some bi-infinite triangular prism and reflections in them would generate an infinite group). Hence, from proposition \ref{vPi} every $3$ faces corresponding to such $3$ generators share a vertex.
\par Now, assume \emph{a contrario} that there is a subset of $S$ of cardinality $4$. Then, from the above, every $3$ of the $4$ faces of $\Pi$ corresponding to those $4$ generators, have a common vertex. On the other hand, there is no common vertex of all those $4$ faces, because otherwise we would have a polyhedral angle with more than $3$ faces and non-obtuse dihedral angles, which is impossible. One can easily see that then $\Pi$ must be a compact tetrahedron, but then the group generated by those $4$ faces would be infinite, a contradiction.
\end{proof}
\subsection{Growth of right-angled cocompact reflection groups in $\HHH$}
\begin{clm}
Let $G$ be a (Coxeter) reflection group of a right-angled compact polyhedron $\Pi$ in $\HHH$ with the standard generating set $S$ and let $L$ be the nerve of $(G,S)$ (see def.~\ref{nerve}). Then $L$ is a flag\footnote{A simplicial complex is \emph{flag} iff any finite set of its vertices which are pairwise connected by edges spans a simplex.} triangulation of $\SS$ (i.e.~its geometric realisation is homeomorphic to $\SS$).
\end{clm}
\begin{proof}
Due to claim \ref{no4inL}, there is no $3$-simplex (nor higher-dimensional ones) in $L$. Now, each pair of faces of $\Pi$ corresponds to a spherical pair of generators iff those faces neighbour (for the ``only if'' part, use proposition \ref{ePi} and the fact that reflections in two disjoint planes generate an infinite group). Similarly, each $3$ faces correspond to a spherical subset of $S$ iff they share a vertex---see the proof of claim \ref{no4inL}. Note that the degrees of vertices of $\Pi$ are all equal to $3$, because the only possibility for a polyhedral angle with right dihedral angles is a trihedral angle (we then call $\Pi$ simple). So, $L$ is a complex dual to the polygonal complex of faces of $\Pi$ (meaning that vertices of $L$ correspond to faces of $\Pi$, edges of $L$---to edges of $\Pi$ and triangles of $L$---to vertices of $\Pi$), hence, it is a triangulation of $\SS$. It remains to show that it is flag: if it were not, then we would have three faces pairwise neighbouring, but not sharing a vertex, whoch would contradict Andreev's theorem (see \cite[thm.~6.10.2 (ii)]{Dav}), as $\Pi$ is right-angled and simple.
\end{proof}
\begin{df}\label{grser}
Let $G$ be a group with finite generating set $S$. Then the \emd{growth series} of $G$ with respect to $S$ is the formal power series $W$ defined by
$$W(z)=\sum_{n=0}^\infty\#S(n)z^n = \sum_{g\in G}z^{l(g)},$$
where for $n\in\mathbb{N}$, $S(n)=B(n)\sm B(n-1)$ is the (graph-theoretic) sphere in the Cayley graph of $(G,S)$, centred at some fixed vertex, of radius $n$, and $l$ is the length function on $(G,S)$ (cf.~definitions \ref{dfgr} and \ref{lgth}).
\end{df}
Growth rate of a group with finite set of generators is exactly the reciprocal of radius of convergence of the growth series of the group. That and the above claim lead to the below theorem.
\begin{thm}[for the proof, see \protect{\cite[example 17.4.3.]{Dav}} with the exercise there]
\label{gr}
For $G$ a (Coxeter) reflection group of a right-angled compact $k$-hedron $\Pi$ in $\HHH$ with the standard generating set $S$, its growth rate
$$\gr(G,S)=\frac{k-4+\sqrt{(k-4)^2-4}}{2}.$$
\end{thm}
\subsection{The proof}
I recall the theorem to be proved:
\begin{thm*}[recalled theorem \ref{3phbasic}]
If $\Pi$ is a compact right-angled polyhedron in $\HHH$ with $k\ge 18$ faces, then for $\CG$ the Cayley graph of reflection group of $\Pi$ with the standard generating set, we have
$$\pc(\CG)<\pu(\CG)$$
for bond and site Bernoulli percolation on $\CG$.
\end{thm*}
\begin{proof}
First, note that, due to remark \ref{g<=gs<=rt}, it is sufficient to show that $\rt(\CG)<\gr(\CG)$ to prove the theorem.
\par Let us put $b_1(k):=2\sqrt{3(k-3)}$ (the upper bound for $\rt(\CG)$) and $b_2(k):=\frac{1}{2}(k-4+\sqrt{(k-4)^2-4})$ (the formula for $\gr(\CG)$). It is sufficient to prove that for real $k\ge 18$,
$$b_1(k)<b_2(k).$$
That will follow once shown for $k=18$, provided that inequality
\begin{equation*}
\dk b_1(k) \le \dk b_2(k)
\end{equation*}
is shown for $k\ge 18$.
\par For $k=18$, we have
$$b_1(k)=2\sqrt{45}<7+\sqrt{48}=b_2(k)$$
and for $k\ge 18$, we differentiate:
\begin{equation*}
\dk b_1(k)=\sqrt{\frac{3}{k-3}}\le 1
\end{equation*}
and
\begin{equation*}
\dk b_2(k) = \frac{1}{2} + \frac{1}{2}\frac{k-4}{\sqrt{(k-4)^2-4}}\ge 1 \ge \dk b_1(k), \label{gr'>=1},
\end{equation*}
which finishes the proof.
\end{proof}
\section{The general case}\label{secgen}
Recall hypothesis \ref{PiG}.
\begin{thm*}[recalled theorem \ref{3phgen}]
If $k\ge13$, then for $\CG$ the Cayley graph of a reflection group of a Coxeter $k$-hedron in $\HHH$ with the standard generating set, we have
$$\pc(\CG)<\pu(\CG)$$
for bond and site Bernoulli percolation on $\CG$.
\end{thm*}
The tools which improve the result in theorem \ref{3phbasic} to the above one (and also to theorem \ref{3phRACpt}), are the upper bound for $\gs(\CG)$, along with a more appropriate upper bound for $\rt(\CG)$, and a lower bound for $\gr(\CG)$, stated below. (Along with the upper bound for $\gs(\CG)$, I establish the equality in the second part of the below theorem, but I do not use it in this \work.)
\begin{thm}\label{thmgs}
Let $\CG$ be an arbitrary regular graph of degree $k\ge 2$ (not necessarily simple) with distinguished vertex $o$ and let $\rt=\rt(\CG,o)$, $\gs=\gs(\CG,o)$. Then
\begin{equation}
\gs \le \frac{\rt+\sqrt{\rt^2-4(k-1)}}{2}.
\end{equation}
If, in addition, $\rt(\CG)>2\sqrt{k-1}$ (e.g.~when $\CG$ is vertex-transitive and simple 
and is not a tree), then the estimate becomes an identity:
\begin{equation}
\gs = \frac{\rt+\sqrt{\rt^2-4(k-1)}}{2}.
\end{equation}
\end{thm}
The proof of the above theorem is delayed to appendix \ref{appxgs}.
\begin{thm}\label{gr>=}
For $\CG$ the Cayley graph of a reflection group of a Coxeter $k$-hedron in $\HHH$ with the standard generating set,\footnote{As in hypothesis \ref{PiG}.} if we assume that $k\ge 6$, then the growth rate of $\CG$
$$\gr(\CG)\ge\frac{k-4+\sqrt{(k-4)^2-4}}{2}.$$
\end{thm}
The proof is presented in appendix \ref{appxgr}.
\begin{lem}\label{rhogen}
For $k$ and $\CG$ as in hypothesis \ref{PiG}, we have
$$\rt(\CG)\le\frac{k+17}{3}.$$
\end{lem}
\begin{proof}
In order to bound the sum \eqref{f(r,q)2} (and to use the corollary \ref{Gabcor}), I put $c_1=c_2=3,c_3=2$ and I am going to bound $q_3(v)$.
\begin{lem}\label{q_3gen}
For $o\neq v\in V(\CG)$, we have
$$q_3(v)\le
\{\begin{array}{r@{\text{, if }}l}
0& r(v)=1\\
2& r(v)=2\\
3& r(v)=3
\end{array}\right.$$
\end{lem}
\begin{denot}
For $e\in\oE(\CG)$, let $|e|\in E(\CG)$ denote $e$ as unoriented edge and let us denote by $H(e)$ the half-space bounded by the plane of $e\d$, containing $e_-$.
\par For $e\in E(\CG)$ (or $e\in \oE(\CG)$), let $R_e$ be the reflection in the plane of $e\d$. Whenever I write $R_e(e')$ for some $e'=(v,w)\in\oE(\CG)$, I mean the oriented edge $(R_e(v),R_e(w))$.
\end{denot}
\begin{proof}
Let $e\in O(\CG)$ be an arbitrary edge passing from $v$ and $N$ be the set of unoriented edges passing to $v$ (when oriented according to $O(\CG)$) traversing faces of $\Pi$ neighbouring $e\d$:
$$N=\{|f|:f\in O(\CG), f_+=v, f\d\sim e\d\}.$$
Call the vertex $e_+$ by $v'$ and let $N'=R_e(N)$ and note that $v'\Pi=R_e(v\Pi)$.
\begin{clm}\label{eN'}
All the edges passing to $v'$ belong (without orientations) to $\{|e|\}\cup N'$.
\end{clm}
\begin{proof}
The proof is by contraposition: for arbitrary $e'\in \oE(\CG)$ passing from $v'$ such that $|e'|\notin\{|e|\}\cup N'$, I show that $e'\in O(\CG)$. So, let $e'\neq \bar e$ be arbitrary edge from $\oE(\CG)$ passing from $v'$ (this orientation is assumed just for convenience), such that $|e'|\notin\{|e|\}\cup N'$. Then,
\begin{itemize}
\item if $e'\d$ neighbours $e\d$, then $R_e(e')\d\sim e\d$ as well and $|R_e(e')|\notin N$, so $|R_e(e')|$ passes from $v$, i.e.~$R_e(e')\in O(\CG)$, so
\begin{equation*}
H(e')\supseteq H(R_e(e'))\cap H(e)\ni o,
\end{equation*}
because $v'\Pi$ has only non-obtuse dihedral angles, so $e'\in O(\CG)$;
\item if $e'\d$ does not neighbour $e\d$, then from proposition \ref{ePi} the planes containing them are disjoint, so $H(e')\supseteq H(e)\ni o$ and $e'\in O(\CG)$.
\end{itemize}
That shows that the edges passing to $v'$ all belong to $\{|e|\}\cup N'$ (without orientations).
\end{proof}
Now, assume that $r(v')=3$. By the claim, two of the edges passing to $v'$ are in $N'$ and the third is $e$ and by proposition \ref{r<=3} the corresponding faces of $v'\Pi$ share a common vertex. So do their reflections in $R_e$, which are $e\d$ and two of the faces corresponding to edges passing to $v$.
\par That means that:
\begin{itemize}
\item if $r(v)=1$, then there is no $e$ as above (i.e.~with $r(e_+)=3$) and $q_3(v)=0$;
\item if $r(v)=2$, then the two faces corresponding to edges passing to $v$ have to lie on intersecting planes, hence share one edge (from prop.~\ref{ePi}) and there are at most two possibilities for such a common vertex with $e\d$ as above (and hence for $e$), so $q_3(v)\le 2$;
\item if $r(v)=3$, then the faces corresponding to edges passing to $v$ share a common vertex and at most three others pairwise---the latter are the only possibilities for a common vertex with $e\d$ as above, so $q_3(v)\le 3$.
\end{itemize}
That finishes the proof of lemma \ref{q_3gen}
\end{proof}
\begin{rem}
Consider vertex $o$. Each neighbour $v$ of $o$ has $r(v)=1$ (because going closer to $o$ from $v$ we must return to $o$). Hence, $q_1(o)=k$, $q_2(o)=q_3(o)=0$.
\end{rem}
Now, for $o\neq v\in V(\CG)$, because $q_1(v)+q_2(v)=k-r(v)-q_3(v)$, we have
\begin{align*}
\textrm{if $r(v)=1$, then} &\quad f_v(\bar c)=3 + \frac{q_1(v)+q_2(v)}{3} + \frac{0}{2}=\frac{k+8}{3},\\
\textrm{if $r(v)=2$, then} &\quad f_v(\bar c)=2\cdot 3 + \frac{q_1(v)+q_2(v)}{3} + \frac{q_3(v)}{2}\le 6+\frac{k-4}{3}+\frac{2}{2}=\frac{k+17}{3},\\
\textrm{if $r(v)=3$, then} &\quad f_v(\bar c)=3\cdot 2 + \frac{q_1(v)+q_2(v)}{3} + \frac{q_3(v)}{2}\le 6+\frac{k-6}{3}+\frac{3}{2}=\frac{k+16\frac{1}{2}}{3},
\end{align*}
and for $v=o$, $f_o(\bar c)=\frac{k}{3}$, so we put $C_f=\frac{k+17}{3}$ in corollary \ref{Gabcor} and obtain the lemma \ref{rhogen}.
\end{proof}
\begin{proof}[Proof of theorem \ref{3phgen}]
First of all, note that, due to remark \ref{g<=gs<=rt}, it is sufficient to show that $\gs(\CG)<\gr(\CG)$ in order to prove the theorem.
\par The calculations are analogous to those in proof of theorem \ref{3phbasic}. Theorems \ref{thmgs}, \ref{gr>=} and lemma \ref{rhogen} give us the following bounds, which I denote by $b_1(k)$, $b_2(k)$, respectively:
\begin{align*}
\gs(\CG) \le \frac{\frac{k+17}{3}+\sqrt{(\frac{k+17}{3})^2-4(k-1)}}{2} &= b_1(k)\\
\gr \ge \frac{k-4+\sqrt{(k-4)^2-4}}{2} &= b_2(k),
\end{align*}
hence it is sufficient to prove the inequality $b_1(k)<b_2(k)$ for $k\ge 13$.
Let us check it for $k=13$:
\begin{equation*}
\rt\le 10,\quad\textrm{so}\quad
b_1(k) = 5+\sqrt{13} < \frac{9+\sqrt{77}}{2} = b_2(k).
\end{equation*}
and, again, check the inequality
$$\dk b_1(k) \le \dk b_2(k)$$
for real $k\ge 13$. The right-hand side derivative is $\ge1$ by \eqref{gr'>=1} and the left-hand side:
\begin{equation*}
\dk\left(\frac{k+17+\sqrt{k^2-2k+325}}{6}\right) = \frac{1}{6}\left(1+\frac{2k-1}{2\sqrt{k^2-2k+325}}\right) \le \frac{1}{6}\left(1+\frac{2k-1}{2k-2}\right) < 1,
\end{equation*}
which finishes the proof.
\end{proof}
\section{The compact right-angled case}\label{secRACpt}
In this section I assume that $\Pi$ is right-angled and compact, hence $G$ is a right-angled Coxeter group acting cocompactly on $\HHH$.
\begin{exmp}\label{dodecah}
One of the simplest examples of such $\Pi$ is the right-angled regular dodecahedron. Its orbit under $G$ is a regular tiling of $\HHH$.
\end{exmp}
\begin{thm*}[recalled theorem \ref{3phRACpt}]
If $\Pi$ is a compact right-angled polyhedron in $\HHH$, then for $\CG$ the Cayley graph of reflection group of $\Pi$ with the standard generating set, we have
$$\pc(\CG)<\pu(\CG)$$
for bond and site Bernoulli percolation on $\CG$.
\end{thm*}
The proof is analogous to the proof of theorem \ref{3phgen}, but I don't use theorem \ref{gr>=} (rather theorem \ref{gr}) and I use an additional fact:
\begin{lem}\label{Delta}
In the setting of the above theorem, $\Delta\le\frac{k-1}{2}$.
\end{lem} 
\begin{proof}
Take any face $F$ of $\Pi$ with $\Delta$ sides. Each edge of $F$ belongs to unique face $F'_i$ other than $F$, where $i=1,\ldots,\Delta$ numerates these edges. There are also $\Delta$ edges of $\Pi$ perpendicular to $F$, incident to its vertices. Outside $F$, at the ends of those edges there are attached faces $F''_i,i=1,\ldots,\Delta$ perpendicular to those edges, respectively.
\par Now:
\begin{itemize}
\item the faces $F,F'_1,\ldots,F'_\Delta$ are pairwise distinct, because no two edges of $F$ lie on a common line;
\item the faces $F''_i,i=1,\ldots,\Delta$ are pairwise distinct, because each of the planes containing them determines uniquely the closest point in the plane of $F$ (disjoint with them) and those points---the vertces of $F$---are pairwise distinct;
\item any face $F''_i,i=1,\ldots,\Delta$ is distinct from any face $F,F'_1,\ldots,F'_\Delta$, because the former are disjoint with $F$ and the latter are not.
\end{itemize}
So all the faces $F,F'_1,\ldots,F'_\Delta,F''_1\ldots,F''_\Delta$ are pairwise distinct, which shows that $k\ge2\Delta+1$ and finishes the proof of the lemma.
\end{proof}
\begin{rem}\label{k>=12}
In the setting of theorem \ref{3phRACpt}, $k\ge 12$. To see that, note first that $k\ge 11$, because a face of $\Pi$, which is right-angled, must have at least $5$ sides, so in lemma \ref{Delta} $\Delta\ge 5$. Now, assume \emph{a contrario} that $k=11$. Then from lemma \ref{Delta} $\Delta=5$, so all faces of $\Pi$ are pentagons. Since $\Pi$ is compact, we obtain twice the number of its edges counting all sides of each of the faces. That gives $5k$, so $k$ must be even, a contradiction.
\end{rem}
\begin{lem}\label{rhoRACpt}
In the setting of theorem \ref{3phRACpt}, we have
$$\rt(\CG)\le \frac{k}{2}+3\frac{1}{10}.$$
\end{lem}
\begin{proof}
Note that every vertex of $v\Pi$ is adjacent to exactly $3$ faces of $v\Pi$, because a polyhedral angle with all dihedral angles right must be trihedral. For $o\neq v\in V(\CG)$, consider the faces of $v\Pi$ traversed by edges passing to $v$. Then, from lemma \ref{Delta}, the number of other faces of $v\Pi$ neighbouring them is at most
$$\{
\begin{array}{l@{\text{ if }}l}
\Delta\le\frac{k-1}{2},& r(v)=1\\
2+2(\Delta-3)=2\Delta-4\le k-5,& r(v)=2\\
\min(3+3(\Delta-4),k-r(v))\le k-3,& r(v)=3\\
\end{array}\right.$$
and so is $q_2(v)+q_3(v)$ (because of claim \ref{eN'}---as in proof of lemma \ref{q_3gen}).
Now, let $c_1=5,c_2=2,c_3=1$. For $o\neq v\in V(\CG)$, basing on all that and on lemma \ref{q_3gen} itself, I estimate (similarly to proof of lemma \ref{rhogen}; note that below, taking first the smallest possible value for $q_1(v)=k-r(v)-q_2(v)-q_3(v)$, then the smallest possible value for $q_2(v)$, indeed gives the upper bounds, because $c_1>c_2>c_3$):
\begin{align*}
\textrm{if $r(v)=1$, then} &\quad f_v(\bar c)=5 + \frac{q_1(v)}{5} +  \frac{q_2(v)}{2} + 0 \le 5 + \frac{\frac{k-1}{2}}{5} +  \frac{\frac{k-1}{2}}{2} = \frac{7k+93}{20},\\
\textrm{if $r(v)=2$, then} &\quad f_v(\bar c)=2\cdot 2 + \frac{q_1(v)}{5} +  \frac{q_2(v)}{2} + q_3(v) \le 4+\frac{3}{5}+\frac{k-7}{2}+2 = \frac{10k+62}{20},\\
\textrm{if $r(v)=3$, then} &\quad f_v(\bar c)=3\cdot 1 + \frac{q_1(v)}{5} +  \frac{q_2(v)}{2} + q_3(v)\le 3+\frac{k-6}{2}+3 = \frac{10k+60}{20}
\end{align*}
and for $v=o$, $f_o(\bar c)=\frac{k}{5}=\frac{4k}{20}$, so we put $C_f=\frac{10k+62}{20}$ in corollary \ref{Gabcor}, as it is the largest of the bounds above (because $k\ge 12$, due to rem.~\ref{k>=12}), and we obtain the lemma.
\end{proof}
\begin{proof}[Proof of the theorem \ref{3phRACpt}]
First, note that, due to remark \ref{g<=gs<=rt}, it is sufficient to show that $\gs(\CG)<\gr(\CG)$ to prove the theorem.
\par Because the conclusion is shown for $k>12$ in general case in theorem \ref{3phgen}, it is sufficient to show it for $k=12$ (because of remark \ref{k>=12}).
So I calculate:
\begin{equation*}
\rt(\CG) \le 9\frac{1}{10},\quad\textrm{so}\quad
\gs(\CG) \le \frac{91+\sqrt{3881}}{20} < 4+\sqrt{15} = \gr(\CG).
\end{equation*}
which shows the desired inequality.
\end{proof}
\section{Remarks on using particular tools}\label{secrems}
Recall hypothesis \ref{PiG}.
\par In the below table I give the conditions (lower bounds) for $k$, necessary and sufficient for obtaining the inequality $\pc(\CG)<\pu(\CG)$ by means of lemma \ref{rhobasic}, \ref{rhogen} and \ref{rhoRACpt}, respectively, and using the value $\gs(\CG)$ or only $\rt(\CG)$, respectively.
\begin{center}
\begin{tabular}{|l||c|c|}
\hline
&Using $\rt$&Using $\gs$\\\hline\hline
Using lem.~\ref{rhobasic}&$k\ge 18$&$k\ge 15$\\\hline
Using lem.~\ref{rhogen}&$k\ge 15$&$k\ge 13$\\\hline
Using lem.~\ref{rhoRACpt}&$k\ge 15$&$k\ge 12$\\\hline
\end{tabular}
\end{center}
One can see that in the above cases th use of $\gs$ improves the bound on $k$ by $2$ or $3$. A similar effect has using lemma \ref{rhogen} instead of \ref{rhobasic} (both valid for the general case from hypothesis \ref{PiG}). Using lemma \ref{rhoRACpt} (proved here only for compact right-angled $\Pi$) plays a role only for $k=12$, when using $\gs$ (which covers the case of example \ref{dodecah}).
\appendix
\section{Estimate for $\gs$}\label{appxgs}
Below, I am going to prove the relations between $\gs$ and $\rt$ in the theorem \ref{thmgs}:
\begin{thm*}[recalled theorem \ref{thmgs}]
Let $\CG$ be an arbitrary regular graph of degree $k\ge 2$ (not necessarily simple) with distinguished vertex $o$ and let $\rt=\rt(\CG,o)$, $\gs=\gs(\CG,o)$. Then
\begin{equation*}
\gs \le \frac{\rt+\sqrt{\rt^2-4(k-1)}}{2}.
\end{equation*}
If, in addition, $\rt(\CG)>2\sqrt{k-1}$ (e.g. when $\CG$ is vertex-transitive and simple and is not a tree), then the estimate becomes an identity:
\begin{equation*}
\gs = \frac{\rt+\sqrt{\rt^2-4(k-1)}}{2}.
\end{equation*}
\end{thm*}
\begin{proof}
Recall notations \ref{Cas}.
The below proposition is basic for obtaining the estimate for $\gs$.
\begin{prop}
Let $\CG$ be any (undirected) regular graph of degree $k\ge 2$ (not necessarily simple). Let choose its vertex $o$ for being its origin. Now, I consider the values $C_n(\CG,o)$ and $a^*_n(\CG,o)$---let us call them $C_n$, $a^*_n$, respectively, for short. Then
\begin{equation*}
C_n = \sum_{d=0}^n a^*_d c^{T_k}(n,d),
\end{equation*}
where for any natural $n$, $c^{T_k}(n,d)$ is the number of paths of length $n$ in $T_k$ joining two points with distance $d$ between them (note that $c^{T_k}(n,d)$ is well-defined).
\end{prop}
\begin{proof}
Let us consider the universal cover%
\footnote{Universal cover of a graph $\CG$ (not necessarily simple) is a graph $\tilde \CG$ with mapping $p:\tilde \CG\to \CG$, i.e.~$p$ maps vertices and edges of $\tilde\CG$ to vertices and edges of $\CG$, respectively, such that for any incident vertex $v$ and edge $e$ of $\tilde\CG$, $p(v)$ and $p(e)$ are also incident, and for any $v\in V(\CG)$ and $\tilde v\in V(\tilde\CG)$ with $p(\tilde v)=v$, the set of edges incident with $\tilde v$ maps bijectively to the set of edges incident with $v$.}%
 of $\CG$, which is the infinite $k$-regular tree $T_k$ (with origin being a vertex $\ot$ chosen to cover $o$). Then every path in $\CG$ starting from $o$ lifts to a unique path starting from $\ot$ in $T_k$ which covers it. In particular, every cycle in $\CG$ starting from $o$ lifts to a unique path in $T_k$ joining $\ot$ to some vertex $\ot'$ covering $o$. So, for any natural $n$, if $d(\ot,\ot')$ is the distance between $\ot$ and $\ot'$, then
\begin{equation}
C_n = \sum_{\ot'\in T_k\text{ covering }o}c^{T_k}(n,d(\ot,\ot')) = \sum_{d=0}^n\#\{\ot'\in T_k\text{ covering $o$}: d(\ot,\ot')=d\}c^{T_k}(n,d),\label{eqprfCas}
\end{equation}
Now, the number of $\ot'\in T_k$ at distance $d$ from $\ot$, covering $o$ is exactly the number of geodesic segments of length $d$ from $\ot$ to such vertices $\ot'$. But a cycle $P$ in $\CG$ lifts to a geodesic in $T_k$ if and only if $P$ does not admit any backtracks. Hence
$$C_n=\sum_{d=0}^n a^*_d c^{T_k}(n,d).$$
\end{proof}
\begin{denot1}
  For any numerical sequence $x=(x_n)_{n=0}^\infty$, denote by $\F(x)$ its generating function, i.e. power series (of $z\in\mathbb{C}$)
$$\F(x)(z)=\sum_{n=0}^\infty x_n z^n.$$
\end{denot1}
\begin{rem}
In this appendix, I consider every sum of a power series as a sum (a number, if it is convergent, or $\pm\infty$, if it diverges to $\pm\infty$), rather than value of some holomorphic continuation of it, unless indicated otherwise.
\end{rem}
We calculate $\F(C)(r)$ for $r\ge 0$:
\begin{equation*}
\F(C)(r) = \sum_{n=0}^\infty C_n r^n = \sum_{n=0}^\infty \sum_{d=0}^n a^*_d c^{T_k}(n,d) r^n = \sum_{d=0}^\infty a^*_d \sum_{n=d}^\infty c^{T_k}(n,d) r^n.
\end{equation*}
As $c^{T_k}(n,d)=0$ for $n<d$, and $c^{T_k}(n,d)/k^n$ is the probability of passing from $\ot$ to some fixed $\ot'$ at distance $d$ from $\ot$ in the simple random walk on $T_k$ in $n$ steps, we have for $z\in\mathbb{C}$
$$\sum_{n=d}^\infty c^{T_k}(n,d) z^n = \sum_{n=0}^\infty (c^{T_k}(n,d)/k^n)(kz)^n = G(d,kz),$$
where $G(d,\cdot)$ is the Green function\footnote{As a power series. For definition, see \cite[1.6]{Woe}.} for the simple random walk on $T_k$ and a pair of its vertices at distance $d$. From lemma 1.24 from \cite{Woe} we have
\begin{align}
G(d,kz) &= \frac{2(k-1)}{k-2+\sqrt{k^2-4(k-1)(kz)^2}}\left(\frac{k-\sqrt{k^2-4(k-1)(kz)^2}}{2(k-1)kz}\right)^d=\notag\\
&= \underbrace{\frac{2(k-1)}{k-2+k\sqrt{1-4(k-1)z^2}}}_{A(z)}\left(\underbrace{\frac{1-\sqrt{1-4(k-1)z^2}}{2(k-1)z}}_{f(z)}\right)^d,\label{Green}
\end{align}
(where we always take the standard branch of square root\footnote{Note that here, under the square roots, if $z$ is close to $0$, then we have values close to $1$}, with $\sqrt{1}=1$). Let us introduce notations $A(z),f(z)$, as indicated in the above formula. For $z=0$, formally, there is a problem with $f(z)$ defined by such formula, but
$$f(z) = \frac{2z}{1+\sqrt{1-4(k-1)z^2}}$$
(for $z$ such that it exists), so if we put $f(0)=0$, then formula \eqref{Green} is satisfied, because $G(d,0)=0^d$. Then the right-hand side of \eqref{Green} is holomorphic on the $0$-centred open ball $B(1/2\sqrt{k-1})$, so the equality holds on that ball. Hence, for $r\in [0;1/2\sqrt{k-1})$,
\begin{equation}
\F(C)(r) = A(r) \sum_{d=0}^\infty a^*_d f(r)^d = A(r) F(a^*)(f(r))
\end{equation}
Note that here $0<A(r)<\infty$, so for $r$ as above,
\begin{equation}\label{FCfinFAfin}
F(C)(r)<\infty \Leftrightarrow F(a^*)(f(r))<\infty.
\end{equation}
\begin{clm}
Put $f(1/2\sqrt{k-1})=1/\sqrt{k-1}$.\footnote{Actually, it follows from \eqref{Green}, although $1/2\sqrt{k-1}$ is a singularity of $f$ as a holomorphic function, so I am avoiding a doubt.} Then $f:[0;1/2\sqrt{k-1}] \to [0;1/\sqrt{k-1}]$ is strictly increasing and onto.
\end{clm}
\begin{proof}
The strict monotonicity is obvious, as well as the continuity. Since $f(0)=0$ and $f(1/2\sqrt{k-1})=1/\sqrt{k-1}$, being onto follows from the Darboux property.
\end{proof}
\begin{rem}
Further in this appendix, I restrict $f$ to $[0;1/2\sqrt{k-1}]$ by default.
\end{rem}
\begin{denot1}\label{radconvrg}
For power series $F(z)=\sum_{n=0}^\infty x_n z^n$, by $\r(F)$ or $\r(x)$ I mean its radius of convergence (equal to $1/\mathrm{gr}((x_n)_n)$)
and by $\R(F)$ or $\R(x)$ I mean the set $\{r\ge 0: F(r)<\infty\}$.
\end{denot1}
\begin{clm}
$f(\r(C))=\min(\r(a^*),1/\sqrt{k-1})$.
\end{clm}
\begin{rem}
The left-hand side above makes sense, because $\r(C)=1/\rt \le 1/2\sqrt{k-1}$ (this bound can be obtained by noting that for $n\in\mathbb{N}$, $C_n(\CG,o)\ge c^{T_k}(n,0)=C_n(T_k,\ot)$---from \eqref{eqprfCas}---hence $\rt(\CG)\ge\rt(T_k)$ and $\rt(T_k)=2/\sqrt{k-1}$ by \cite[lem.~1.24]{Woe}).
\end{rem}
 \begin{proof}
 I am going to split the proof of the equality into bounding left-hand side by right-hand side and \emph{vice-versa}:
 \begin{enumerate}
 \item Proof of $f(\r(C)) \le \min(\r(a^*),1/\sqrt{k-1})$: Let $w\ge 0$. Assume that $w<f(\r(C))$ (hence $w<1/\sqrt{k-1}$). Then $f^{-1}(w)<\r(C)$ and $F(C)(f^{-1}(w))<\infty$, so from \eqref{FCfinFAfin} $F(a^*)(w)<\infty$ (as $f^{-1}(w)<1/2\sqrt{k-1}$) and so $w\le \r(a^*)$, which finishes this part of the proof.
 \item Proof of $f(\r(C)) \ge \min(\r(a^*),1/\sqrt{k-1})$: Let $w\ge 0$ and assume that $w<\min(\r(a^*),1/\sqrt{k-1})$. Then, if we put $r=f^{-1}(w)<1/2\sqrt{k-1}$, then we have $F(a^*)(f(r))<\infty$, so from \eqref{FCfinFAfin}, $F(C)(r)<\infty$, hence $r\le \r(C)$. That amounts to $w\le f(\r(C))$, which finishes the proof.
 \end{enumerate}
 \end{proof}
\comment{
\begin{proof}
From \eqref{FCfinFAfin} we have
$$[0;1/2\sqrt{k-1})\cap \R(F(C)) = [0;1/2\sqrt{k-1})\cap f^{-1}(\R(F(a^*))),$$
so, taking the images in $f$, we have
$$f[[0;1/2\sqrt{k-1})\cap \R(F(C))] = [0;1/\sqrt{k-1})\cap \R(F(a^*))\cap \im(f) = [0;1/\sqrt{k-1})\cap \R(F(a^*)).$$
Taking the suprema of those sets gives
$$f(\min(1/2\sqrt{k-1},\r(C))) = \min(1/\sqrt{k-1},\r(a^*))$$
and $$f(\r(C)) = \min(1/\sqrt{k-1},\r(a^*)),$$
because $\r(C)\le 1/2\sqrt{k-1}$.
\end{proof}
}
Now,
\begin{equation*}
f(\r(C)) = f(1/\rt) = \frac{2}{\rt+\sqrt{\rt^2-4(k-1)}},
\end{equation*}
hence
\begin{equation*}
\gs = \frac{1}{\r(a^*)} \le \frac{1}{f(\r(C))} = \frac{\rt+\sqrt{\rt^2-4(k-1)}}{2},
\end{equation*}
which proves the first part of theorem \ref{thmgs}. For the second part, assume that $\rt>2\sqrt{k-1}$. Then
$$\min(\r(a^*),1/\sqrt{k-1}) = f(\r(C)) < 1/\sqrt{k-1},$$
so $\r(a^*)<1/\sqrt{k-1}$, hence
$$f(\r(C)) = \r(a^*),$$
which, as above, gives the equality
$$\gs = \frac{\rt+\sqrt{\rt^2-4(k-1)}}{2}.$$
\end{proof}
\section{Estimate for growth rate of a Coxeter reflection group in $\HHH$}\label{appxgr}
In this appendix I prove theorem \ref{gr>=}. Recall the hypothesis \ref{PiG} and notions from subsection \ref{geomPiG}.
\begin{thm*}[recalled theorem \ref{gr>=}]
For $k$ and $\CG$ as in hypothesis \ref{PiG}, if we assume that $k\ge 6$, then the growth rate of $\CG$
$$\gr(\CG)\ge\frac{k-4+\sqrt{(k-4)^2-4}}{2}.$$
\end{thm*}
\begin{proof}
Let $W$ be the growth series of $G$ with respect to $S$ (see def.~\ref{grser}).
I use the below formula of Steinberg for $1/W(z)$.
\begin{thm}[Steinberg, \protect{\cite[1.28]{Stein}}, see also \protect{\cite[sect.~17.1]{Dav} or \cite[thm.~1]{Kolp}}]\label{Steinb}
$$\frac{1}{W(z)}=\sum_{T\in\mathcal{F}} \frac{(-1)^{\#T}}{W_T(z^{-1})}$$
(as formal Laurent series of $z$), where $\mathcal{F}=\{T\subseteq S:\langle T\rangle\text{ is finite}\}$ (i.e.~the family of spherical subsets of $S$) and $W_T$ is the growth series of the Coxeter group $\langle T\rangle$ with respect to generating set $T$.
\end{thm}
\begin{rem}
The above theorem is formulated in terms of formal Laurent series, as in \cite{Stein} (see notations 1.24 there). (Note that it makes sense, because every formal Laurent series admits formal Laurent series reciprocal to it, and for $T\in\mathcal{F}$, $W_T$ are polynomials.) On the other hand, it can be also viewed as a meromorphic function on $\mathbb{C}$, because due to \cite[1.26]{Stein}, it is a rational function. So for the rest of this appendix, I treat all growth series as meromorphic functions.
Also, I am going to use the convention: $1/0=\infty$, $1/\infty=0$.
\end{rem}
For a power series $\sum_{n=0}^\infty x_n z^n$, let $\mathrm{r}(x)$ be its radius of convergence (as in definition \ref{radconvrg}).
\begin{lem}[\protect{\cite[par.~7.21]{Titch}}]\label{poscoeff}
Any power series $\sum_{n=0}^\infty x_n z^n$ with $x_n\ge 0$ has a singularity at $\mathrm{r}(x)$.
\end{lem}
\begin{rem}\label{root1/W}
The radius of convergence of $W$, which is the least positive pole of $W$, due to the above lemma, equals the radius of convergence of
$$\sum_{n=0}^\infty\#B(n)z^n = \sum_{n=0}^\infty \sum_{m=0}^n \#S(m)z^n = W(z)/(1-z),$$
which is $1/\gr(\CG)$. Hence, it is the least positive root of the function $1/W$.
\end{rem}
\begin{denot}\label{L}
Let $L$ be the nerve of $(G,S)$ (as an abstract simplicial complex---see definition \ref{nerve}). Note that $\mathcal{F}$ is the set of all simplices in $L$ plus the empty set, so for $\s\in L$, $W_\s$ should be understood with $\s\subseteq S$. For a simplicial complex $C$ (abstract or geometric one), I denote by $V(C)$, $E(C)$ its sets of vertices and edges, respectively.
Note also that $L$ does not need to be connected.
\end{denot}
In order that the rest of the proof worked, I have to show now the conclusion of the theorem for three exceptional cases, and after that exclude them from consideration (hypothesis \ref{L!=s+vs}). Those cases are:
\begin{itemize}
\item the nerve $L$ contains only vertices (i.e.~no edges);
\item there is only $1$ edge in $L$;
\item there is only $1$ triangle in $L$ and no edges outside that triangle
\end{itemize}
and i consider them at once. Namely, in each of those cases, there is a set $S'\subseteq S$ of $k-3$ isolated vertices of $L$, which generates a subgroup $G'<G$ isomorphic to $\mathbb{Z}_2^{*(k-3)}$ (each of the free factors of $\mathbb{Z}_2^{*(k-3)}$ is generated by some $s\in S'$). Hence, the Cayley graph $\CG'$ of $(G',S')$, which is the infinite $(k-3)$-regular tree (with growth rate $k-4$), embeds in $\CG$, so
$$\gr(\CG) \ge \gr(\CG') = k-4 \ge \frac{k-4+\sqrt{(k-4)^2-4}}{2},$$
which finishes the proof for the three cases above. Now, I exclude those cases:
\begin{hypot}\label{L!=s+vs}
For the rest of the proof, let $L$ contain more that $1$ edge and not have exactly $1$ triangle containing all the edges of $L$. (In other words: $L$ doesn't amount to a collection of isolated vertices and single $0$-, $1$- or $2$-simplex.)
\end{hypot}
Using claim \ref{no4inL} and theorem \ref{Steinb}, I calculate a formula for $1/W(t)$:
\begin{equation}\label{1/W3sums}
\frac{1}{W(t)} = 1 - \sum_{v\text{---a vertex of }L} \frac{1}{W_{\{v\}}(t^{-1})} + \sum_{e\text{---a side of }L} \frac{1}{W_e(t^{-1})} - \sum_{f\text{---a $2$-simplex of }L} \frac{1}{W_f(t^{-1})}.
\end{equation}
\begin{rem}
Now, the main idea of the proof is to use the ,,right-angled compact'' counterpart of $W$. Namely, imagine a right-angled Coxeter group $G\rb$ with generating set $S\rb$ of $k$ elements and with nerve $L\rb$ which is a flag triangulation of $\SS$ (as in theorem \ref{gr}; I use a convention of putting $\cdot\rb$ on elements concerning the ``right-angled compact version'' of $G$). Then $L\rb$ has $k$ vertices; let $f_1\rb$, $f_2\rb$ be the numbers of its edges and triangles, respectively. Then, they are uniquely determined by $k$, using Euler formula for such triangulation:
\begin{align*}
k-f_1\rb+f_2\rb=2.
\end{align*}
We have also $2f_1\rb=3f_2\rb$, so
\begin{align}
2k - 3f_2\rb + 2f_2\rb = 4,\notag\\
f_2\rb = 2(k - 2),\text{\quad and}\label{f_2rb}\\
f_1\rb = \frac{3}{2}f_2\rb = 3(k - 2).\label{f_1rb}
\end{align}
So, if $W\rb$ is the growth series of $(G\rb,S\rb)$, then, using formula \eqref{1/W3sums}, we have
\begin{equation}
\frac{1}{W\rb(t)} = 1 - \frac{k}{t^{-1}+1} + \frac{f_1\rb}{(t^{-1}+1)^2} - \frac{f_2\rb}{(t^{-1}+1)^3},\label{1/Wrbsum}
\end{equation}
as for $\mathbb{Z}_2^n$ with generating set consisting of the generators of the factors (copies of $\mathbb{Z}_2$), its growth series is $(1+z)^n$ (for $n=1$, it is obvious, for other $n$, see \cite[17.1.13]{Dav}); so from \eqref{f_1rb} and \eqref{f_2rb}
\begin{equation}
\frac{1}{W\rb(t)} = \frac{t-1}{(t+1)^3}(-t^2 + (k-4)t-1).\label{1/Wrbprod}
\end{equation}
As $k\ge 6$, all the roots of $1/W\rb(t)$ are non-negative and the least one is
$$\frac{k-4-\sqrt{(k-4)^2-4}}{2}.$$
In this proof, I am not going to use existence of such group $G\rb$, only the fact that the right-hand side of the inequality of the theorem \ref{gr>=} (which coincides with formula for the growth of $(G\rb,S\rb)$) is indeed reciprocal of that least positive root of $1/W\rb(t)$ (similarly as in remark \ref{root1/W}). So I define just the function $1/W\rb(t)$ by \eqref{1/Wrbsum} (in terms of $f_1\rb,f_2\rb$ defined by \eqref{f_1rb} and \eqref{f_2rb}) and I want to bound the least positive root of $1/W$ by that of $1/W\rb$ (from above), which is less or equal to $1$. Note that $1/W(0),1/W\rb(0)=1>0$ and $1/W$, $1/W\rb$ are continuous on $[0;1]$ (which is easily seen from thm.~\ref{Steinb}, formulae \eqref{1/Wrbsum}, \eqref{1/Wrbprod} and for $1/W$ at $0$---from the fact that $\r(W)>0$), so, to get that bound, it suffices to prove that
\begin{equation}\label{ineqWWrb}
1/W(t)\le 1/W\rb(t)
\end{equation}
for $t\in(0;1]$. I do it in claims \ref{clWsum} and \ref{clsumWrb}.
\end{rem}
\begin{prop}
The nerve $L$ embeds into sphere $\SS$ (meaning that its geometric realisation embeds into $\SS$).\footnote{A combinatorial version of the idea of this fact in a bit different setting is present in \cite[example 7.1.4]{Dav}.}
\end{prop}
\begin{proof}
First, note that $\B=\bdi\Pi\cup\bd\Pi$ is homeomorphic to $\SS$. To see this, consider Klein unit ball model of $\HHH$ (see e.g.~\cite[chapter I.6]{BH}). Here, $\Pi$ is a convex set (as a subset of $\mathbb{R}^3$), in particular, it is star-convex  with regard to some point in $\int\Pi$, so indeed, $\B\cong\SS$.
\par Now, i construct an embedding $\Phi$ of $L$ into $\B$ (which completes the proof). Let us fix an interior point $c_f$ or $c_e$, respectively, of each face $f$ and of each edge $e$ of $\Pi$. Then, for $s\in S=V(L)$, I put $\Phi(s)=c_f$, where $f$ is the face of $\Pi$ corresponding to $s$. Next, I embed the edges of $L$: for $e=\{s,t\}\in E(L)$, the planes of faces corresponding to $s$ and $t$ have non-empty intersection (in order that the corresponding reflections generated a finite group), so by proposition \ref{ePi} those faces are neighbours---let $e\d$ denote their common edge\footnote{Not to be confused with $e\d$ from notation \ref{ed}.}. I join $\Phi(s)$ with $c_{e\d}$ and $\Phi(t)$ with $c_{e\d}$ by geodesic segments (lying in the faces corresponding to $s$, $t$, respectively). The union of those two segments is a path in $\B$ joining $\Phi(s)$ with $\Phi(t)$---let $\Phi(e)$ be that path. Note that $\Phi(e_1)$ and $\Phi(e_2)$ are disjoint off the endpoints for edges $e_1\neq e_2$ of $L$, so now $\Phi$ is an embedding of $1$-skeleton of $L$ into $\B$. Now, each $2$-simplex $\s$ of $L$ corresponds to three generators $s_i\in S$ ($i=1,2,3$). Let $e_i, i=1,2,3$ be edges of $\s$. Then points $\Phi(s_i)$ are pairwise joined by paths $\Phi(e_i)$. Because $\langle s_1,s_2,s_3\rangle$ is finite, the planes of faces corresponding to $s_1,s_2,s_3$ have non-empty intersection, due to corrolary II.2.8 from \cite{BH} ($\HHH$ is a $\mathrm{CAT}(0)$ space). So, from proposition \ref{vPi}, those faces share a vertex $p$. From Jordan-Schoenflies theorem, $\bigcup_{i=1}^3\Phi(e_i)$ disconnects $\B$ into two open discs (up to homeomorphism) with boundary $\bigcup_{i=1}^3\Phi(e_i)$. So I take the closure of the component of $\B\sm\bigcup_{i=1}^3\Phi(e_i)$ containing $p$---call it $D$. Obviously, it is an embedding of $\s$ in $\B$ (with sides $\Phi(e_i)$ and vertices $\Phi(s_i)$). Put $\Phi(\s)=D$. Then indeed, $\Phi$ is an embedding of $L$, because for different $2$-simplices $\s_1,\s_2$ of $L$, the interiors (taken in $\B$) of $\Phi(\s_1)$ and $\Phi(\s_2)$ are disjoint.
\end{proof}
\begin{df}\label{L1}
Let us denote by $L_1$ the embedding of $1$-skeleton of $L$ in $\SS$ from the above proposition. I call closure of a component of its complement in $\SS$ its \emd{face}. I am going to identify $1$-skeletons of $L$ and of its embedding in $\SS$. Also, I declare a face of $L_1$ to belong to $L$ (or to be a $2$-simplex of $L$) iff it equals the embedding of some $2$-simplex of $L$.
For $f$---a face of $L_1$, I denote by $\deg(f)$ the \emd{number of sides} of $f$, i.e.~number of edges lying in $f$, but with edges crossing interior of $f$ counted twice.
\end{df}
\begin{denot}\label{f's}
Let $f_0,f_1,f_2$ be the numbers of vertices, edges and faces of $L_1$, respectively, $\ft2$---of $3$-sided faces of $L_1$ and $\fL2$---the number of 2-simplices of $L$. For $e$---an edge of $L$, let $m(e)$ be the order of $s_1s_2$ in $G$, where $e=\{s_1,s_2\}$ (so that the dihedral angle of $\Pi$ at the edge corresponding to $e$ is $\pi/m(e)$). (Note that $f_0=k$.)
\end{denot}
\begin{rem}\label{WvWe}
Note that for a vertex $v$ of $L$,
\begin{equation}\label{Wv}
W_{\{v\}}(z)=z+1
\end{equation}
and, by an easy exercise, for an edge $e$ of $L$,
\begin{equation}
W_e(z)=(z+1)(z^{m(e)-1}+\cdots+z+1).
\end{equation}
\end{rem}
\begin{denot}
To ease our work with such polynomials, for $n\in\mathbb{N}$, I define polynomials
\begin{equation*}
[n](z)=\underbrace{z^{n-1}+\cdots+z+1}_{n\text{ summands}}
\end{equation*}
and for $n_1,\ldots,n_m\in\mathbb{N}$,
\begin{equation*}
[n_1,\ldots,n_m](z)=\prod_{i=1}^m[n_i](z).
\end{equation*}
I will drop ``$(z)$'' if the argument is obvious.
\end{denot}
For $t\in(0;1]$, using this notation, remark \ref{WvWe} and \eqref{1/W3sums}, we have
\begin{align}
\frac{1}{W(t)}&=1 - \frac{f_0}{t^{-1}+1} + \sum_{e\text{---a side of }L} \frac{1}{[2,m(e)](t^{-1})} - \sum_{f\text{---a $2$-simplex of }L} \frac{1}{W_f(t^{-1})}
\intertext{and, reordering the sums into a sum taken over the faces of $L_1$}
&=1 - \frac{f_0}{t^{-1}+1} + \sum_{f\text{---a face of }L_1} \underbrace{\left(\sum_{e\text{---a side of }f}\frac{1}{2}\frac{1}{[2,m(e)](t^{-1})} - \frac{\ind_{f\in L}}{W_f(t^{-1})}\right)}_{A(f)}.\label{A(f)}
\end{align}
It will be convenient to work with the expression $A(f)$ defined above (for $f$ a face of $L_1$).
\begin{rem}\label{remsidessum}
The variable $e$ of the last summation above runs through the sides of $f$, which means that some edge may be counted twice there.
\end{rem}
\begin{clm}\label{clWsum}
For $t\in(0;1]$,
\begin{equation*}
\frac{1}{W(t)} \le 1 - \frac{f_0}{t^{-1}+1} + \frac{f_1}{(t^{-1}+1)^2} - \frac{\ft2}{(t^{-1}+1)^3}.
\end{equation*}
\end{clm}
\begin{proof}
It suffices to prove the bound
\begin{equation}
A(f) \le \underbrace{\frac{\deg(f)}{2(t^{-1}+1)^2} - \frac{\ind_{f\text{---3-sided}}}{(t^{-1}+1)^3}}_{B(f)} \label{B(f)}
\end{equation}
for $f$---a face of $L_1$, to take a sum over all such $f$'s, and to make calculation analogous to deriving \eqref{A(f)} from \eqref{1/W3sums} (with similar reodering of sums) for the right-hand side of the claim (written using summations on edges and $3$-sided faces of $L_1$), to complete the proof.
So I estimate $A(f)$ for $f$ a face of $L_1$, considering three cases:
\begin{case}$f$ is not 3-sided\end{case}
Then
\begin{equation*}
A(f) \le \sum_{e\text{---a side of }f}\frac{1}{2[2,2](t^{-1})} - 0 = B(f).
\end{equation*}
I used here the fact that $[n]\le[m]$ for $n\le m$, on $[0;\infty)$.
\begin{case}$f$ is 3-sided, but outside $L$\end{case}
Note that because of hypothesis \ref{L!=s+vs} $\overline{\SS\sm f}$ is not a triangle of $L$. Hence, vertices of $f$ generate an infinite subgroup of $G$. It means that the dihedral angles between the faces of $\Pi$ corresponding to the vertices of $f$, sum up to a number not greater than $\pi$,\footnote{See e.g.~\cite[exercise 6.8.10]{Dav} combined with \cite[theorem 6.8.12]{Dav} for an explanation.} so
$$\sum_{e\text{---a side of }f}\frac{1}{m(e)}\le 1.$$
Note also that for an edge $e$ and $t^{-1}>1$,
\begin{align}
[2,m(e)](t^{-1}) &= ((t^{-2})^{\frac{m(e)}{2}}-1)(t^{-1}+1)/(t^{-1}-1) \ge \notag\\
&\ge \underbrace{\frac{m(e)}{2}(t^{-2}-1)}_{\text{value of tangent to $z^{\frac{m(e)}{2}}-1$ at $z=1$ taken at $t^{-2}$}}(t^{-1}+1)/(t^{-1}-1) = \frac{m(e)}{2}(t^{-1}+1)^2.
\end{align}
This is also true for $t=1$. Hence
\begin{equation*}
A(f) \le \sum_{e\text{---a side of }f}\frac{1}{m(e)(t^{-1}+1)^2} - 0 \le \frac{1}{(t^{-1}+1)^2} < \frac{3}{2}\frac{1}{(t^{-1}+1)^2} - \frac{1}{2}\frac{1}{(t^{-1}+1)^3} = B(f),
\end{equation*}
as $t^{-1}+1 > 1$.
\begin{case}$f$ is a triangle from $L$\end{case}
Because here $f\subseteq S$ is spherical, the Steinberg formula (theorem \ref{Steinb}) yields
\begin{align}
\frac{1}{W_f(t)} &= 1 - \frac{3}{t^{-1}+1} + \sum_{e\text{---a side of }f} \frac{1}{[2,m(e)](t^{-1})} - \frac{1}{W_f(t^{-1})} =
&= 1 - \frac{3}{t^{-1}+1} + 2A(f) + \frac{1}{W_f(t^{-1})}.\label{Steinb f in L}
\end{align}
Let $m(f)$ be the length of the longest element of $\langle f\rangle$ (i.e.~$m(f)=\max_{g\in\langle f\rangle}l(g)$, using definition \ref{lgth}).
By lemma 17.1.1.~
in \cite{Dav}, we have $W_f(t)=t^{m(f)}W_f(t^{-1})$, so
\begin{align}
\frac{t^{-m(f)}-1}{W_f(t^{-1})} = 1 - \frac{3}{t^{-1}+1} + 2A(f).\label{A(f)(W_f)final}
\end{align}
Consider ``right-angled counterparts'' of $\langle f\rangle$ and $W_f$, which are $\mathbb{Z}_2^3$ and its growth series $W_f\rb$, respectively, the latter given by
\begin{equation}
\frac{1}{W_f\rb(t)} = 1 - \frac{3}{t^{-1}+1} + \frac{3}{(t^{-1}+1)^2} - \underbrace{\frac{1}{(t^{-1}+1)^3}}_{1/W_f\rb(t^{-1})}
\end{equation}
(as in first equality of \eqref{Steinb f in L}; see also explanation of \eqref{1/Wrbsum}). Note that computations analogous to \eqref{Steinb f in L} through \eqref{A(f)(W_f)final} can be made for $1/W_f\rb$ and $B(f)=\frac{3}{2(t^{-1}+1)^2} - \frac{1}{(t^{-1}+1)^3}$ in place of $1/W_f$ and $A(f)$, giving
\begin{equation}
\frac{t^{-3}-1}{(t^{-1}+1)^3} = \frac{t^{-3}-1}{W_f\rb(t^{-1})} = 1 - \frac{3}{t^{-1}+1} + 2B(f).
\end{equation}
So the inequality $A(f)\le B(f)$ is equivalent to
\begin{align}
\frac{t^{-m(f)}-1}{W_f(t^{-1})} &\le \frac{t^{-3}-1}{(t^{-1}+1)^3},
\intertext{which is obvious for $t=1$, and for $t<1$, it is equivalent to}
[m(f),2,2,2](t^{-1}) &\le W_f(t^{-1})[3](t^{-1}).\label{ineqW_fbeftable}
\end{align}
To prove this, I will need the following fact:
\begin{prop}
Let $a,b$ be natural numbers such that $a\le b+1$. Then for any natural $d\le a$,
$$[a-d,b+d](t)\le[a,b](t)$$
for any $t\ge 0$.
\end{prop}
\begin{proof}
Let $t\ge 0$. First, I prove the conclusion for $d=1$:
\begin{align*}
[a,b](t)-[a-1,b+1](t) &= [a,b] - ([a]-t^{a-1})([b]+t^b) = t^{a-1}[b] - [a]t^b + t^{a+b-1} =\\
&= ([a+b-1]-[a-1]) - ([a+b]-[b]) + t^{a+b-1} =\\
&= -t^{a+b-1} - [a-1] + [b] + t^{a+b-1} = [b] - [a-1] \ge 0.
\end{align*}
The general case follows by induction on $d$.
\end{proof}
Below, I use the table 1.~from \cite{KelPer}, giving convenient formulae for $W_f$, which is growth series of a reflection group of a Coxeter triangle on $\SS$, i.e.~of $\mathrm{G}_2^{m(f)-1}\times\mathbb{Z}_2$, where $m(f)\ge 3$, $\mathrm{A}_3$, $\mathrm{B}_3$ or $\mathrm{H}_3$ (with the standard sets of generators). I consider those four cases below (for $t<1$), using the above proposition and the fact that $m(f)=\deg(W_f)$. Here, while $t^{-1}$ is still argument of the polynomials, I drop it for brevity.
\begin{center}
\begin{tabular}{|c|c|l|}
\hline
$\langle f\rangle\cong$ & $W_f$
 & \multicolumn{1}{|c|}{Proof of \eqref{ineqW_fbeftable}}\\ \hline
$\mathrm{G}_2^{m(f)-1}\times\mathbb{Z}_2$ & $[2,2,m(f)-1]$ & $[2,2,2,m(f)]\le[2,2,3,m(f)-1]=[3]W_f$\rule{0pt}{3ex}\\
$\mathrm{A}_3$ & $[2,3,4]$ & $[2,2,2,m(f)]=[2,2,2,6]\le[2,2,3,5]\le[2,3,3,4]=[3]W_f$\rule{0pt}{3ex}\\
$\mathrm{B}_3$ & $[2,4,6]$ & $[2,2,2,m(f)]=[2,2,2,9]\le[2,2,4,7]\le[2,3,4,6]=[3]W_f$\rule{0pt}{3ex}\\
$\mathrm{H}_3$ & $[2,6,10]$ & $[2,2,2,m(f)]=[2,2,2,15]\le[2,2,7,10]\le[2,3,6,10]=[3]W_f$\rule[-2ex]{0pt}{5ex}\\ \hline
\end{tabular}
\end{center}
\par That finishes the proof of inequality \eqref{B(f)} in case of $f$ a triangle from $L$, hence completes the proof of the claim.
\end{proof}
\begin{clm}\label{clsumWrb}
For $t\in(0;1]$,
\begin{equation*}
1 - \frac{f_0}{t^{-1}+1} + \frac{f_1}{(t^{-1}+1)^2} - \frac{\ft2}{(t^{-1}+1)^3} \le \frac{1}{W\rb(t)}
\end{equation*}
(left-hand side above is exactly the right-hand side in claim \ref{clWsum}).
\end{clm}
\begin{proof}
From \eqref{1/Wrbsum}, to prove the claim, it suffices to show that
\begin{equation*}
\frac{f_1}{(t^{-1}+1)^2} - \frac{\ft2}{(t^{-1}+1)^3} \le \frac{f_1\rb}{(t^{-1}+1)^2} - \frac{f_2\rb}{(t^{-1}+1)^3}.
\end{equation*}
Euler formula for $L_1$ gives
\begin{equation}
f_1\le f_0+f_2-2\label{ineqEulL_1}
\end{equation}
(recall that $L_1$ may be disconnected), so the desired inequality above is implied by
\begin{align}
\frac{f_0+f_2-2}{(t^{-1}+1)^2} - \frac{\ft2}{(t^{-1}+1)^3} &\le \frac{3(f_0-2)}{(t^{-1}+1)^2} - \frac{2(f_0-2)}{(t^{-1}+1)^3},\\
\frac{f_2}{(t^{-1}+1)^2} - \frac{\ft2}{(t^{-1}+1)^3} &\le \frac{2(f_0-2)t^{-1}}{(t^{-1}+1)^3}.\label{clsumWrb.ineqf_2}
\end{align}
Now, counting sides in every face of $L_1$ gives
\begin{equation}
2f_1 \;\ge\; 3\ft2 + 4(f_2-\ft2) \;=\; 4f_2-\ft2,
\end{equation}
so from \eqref{ineqEulL_1}
\begin{align}
2(f_0+f_2-2) \;&\ge\; 4f_2-\ft2,\notag\\
2f_0-2f_2-4 \;&\ge\; -\ft2,\label{ineqft2}
\end{align}
so inequality \eqref{clsumWrb.ineqf_2} holds, provided that
\begin{align*}
\frac{f_2}{(t^{-1}+1)^2} + \frac{2f_0-2f_2-4}{(t^{-1}+1)^3} \;&\le\; \left. \frac{2(f_0-2)t^{-1}}{(t^{-1}+1)^3}\quad \right|\cdot(t^{-1}+1)^3\\
f_2(t^{-1}-1) \;&\le\; 2(f_0 - 2)(t^{-1}-1),\\
0 \;&\le\; (2f_0 - 4 - f_2)(t^{-1}-1),
\end{align*}
which is true, because $t^{-1}\ge 1$ and, due to \eqref{ineqft2}, $2f_0-4-f_2 \ge f_2-\ft2 \ge 0$.
\end{proof}
Claims \ref{clWsum} and \ref{clsumWrb} give the inequality \eqref{ineqWWrb} and hence complete the proof of the theorem.
\end{proof}

\end{document}